\documentclass{amsart}
\usepackage{amsfonts,amsmath,amsthm,amssymb}
\usepackage[linesnumbered,ruled,noend]{algorithm2e}
\usepackage{float}
\usepackage{latexsym}
\usepackage{multicol}
\usepackage{color}
\usepackage{pifont}
\usepackage{enumerate}
\usepackage{lmodern}
\usepackage{makecell}
\usepackage{adjustbox}
\usepackage[backref,pdftex=true,colorlinks,linkcolor=blue,citecolor=red,pdfstartview=FitH]{hyperref}
\theoremstyle{plain}
\newtheorem{theorem}{Theorem}[section]

\newtheorem{proposition}[theorem]{Proposition}
\newtheorem{corollary}[theorem]{Corollary}

\newtheorem{result}[theorem]{Result}
\theoremstyle{definition} 
\newtheorem{definition}[theorem]{Definition}

\newtheorem*{example*}{Example}

\newcommand{\Q}{\mathbb{Q}}

\newcommand{\N}{\mathbb{N}}

\begin{document}
\title[]{Nontrivial solutions to the Fermat equation $x^3+y^3=kz^3$ over quadratic number fields}

\author{Alejandro Arg\'{a}ez-Garc\'{i}a}
\address{Facultad de Matemáticas, Universidad Aut\'{o}noma de Yucat\'{a}n. Perif\'{e}rico Norte Kil\'{o}metro 33.5, Tablaje Catastral 13615 Chuburna de Hidalgo Inn, M\'{e}rida, Yucat\'{a}n, M\'{e}xico. C.P. 97200 }
\email{alejandro.argaez@correo.uady.mx}

\author{Javier Diaz-Vargas}
\address{Facultad de Matemáticas, Universidad Aut\'{o}noma de Yucat\'{a}n. Perif\'{e}rico Norte Kil\'{o}metro 33.5, Tablaje Catastral 13615 Chuburna de Hidalgo Inn, M\'{e}rida, Yucat\'{a}n, M\'{e}xico. C.P. 97200  }
\email{javier.diaz@correo.uady.mx}
\date{\today}

\author{Elí Pech-Moreno}
\address{Facultad de Matemáticas, Universidad Aut\'{o}noma de Yucat\'{a}n. Perif\'{e}rico Norte Kil\'{o}metro 33.5, Tablaje Catastral 13615 Chuburna de Hidalgo Inn, M\'{e}rida, Yucat\'{a}n, M\'{e}xico. C.P. 97200  }
\email{luis.eli.pech@gmail.com}
\date{\today}


\begin{abstract}
We give sufficient conditions to determine the existence of nontrivial solutions to the Fermat equation $x^3+y^3=kz^3$ over $\Q(\sqrt{d})$ by constructing a relationship with the points on the elliptic curve $y^2=x^3-432d^3k^2$ over $\Q$ for certain $k\in\N$.
\end{abstract}
\maketitle

\section{Introduction}

Over the past years, the study of Fermat-type equations and their solutions has become an interesting research topic. When taking the classical Fermat equation $x^n+y^n=z^n$, one can wonder how to generalise it. Moreover, how to compute its solutions and how to classify them.

We can generalise it by modifying its exponents or coefficients and changing its base field of definition. In other words, the generalisation is given by the Fermat equation $Ax^p+By^q=Cz^r$ over any field $K$. Particularly, the equation $Ax^3+Ay^3=Cz^3$ over the quadratic number field $K=\Q(\sqrt{d})$ can be rewritten as $x^3+y^3=kz^3$,  still lying on $K$. 

Furthermore, we can map a solution to a Fermat equation into a point on an elliptic curve. In fact, in Section \ref{sec:proof} we prove the following theorem:

\begin{theorem}\label{teo:general}
Let $d$ and $k$ be integer numbers such that $d$ is squarefree and $k>0$ is cubefree. If there is a nontorsion point on $y^2=x^3-432d^3k^2$ over $\Q$, then there exists a nontrivial solution of $x^3+y^3=kz^3$ over $K$.
\end{theorem}

Moreover, we continue our exploration of the case for $k$, a positive cubefree integer, obtaining a series of criteria to determine sufficient conditions to determine the existence of nontrivial solutions to $ x^3+y^3=kz^3$ over $K$.

\begin{result}
Let $d$ and $k$ be integer numbers such that $d$ is a squarefree and $k$ is a positive cubefree, also $d=2^u 3^m d_6$ and $k=2^v 3^n k_6$ with $\gcd(6,d_6)=\gcd(6,k_6)=1$ and $k_6$ a squarefree integer. Let $y^2=x^3-432d^3k^2$ be defined over $\Q$. Assume the Birch and Swinnerton-Dyer conjecture is true when the elliptic curve rank is greater than $1$. Then, we offer sufficient conditions for nontrivial solutions to  $x^3+y^3=kz^3$ to exist over $K$ summarized in Section \ref{sec:sum}.
\end{result}

On the other hand, in the literature, we find two criteria for the existence of nontrivial solutions when $k=1$. In \cite{A}, the existence of nontrivial solutions to the cubic Fermat equation in $K$ is proved based on the number of solutions of certain ternary quadratic forms and conditioned to the Birch and Swinnerton–Dyer conjecture. On the other hand, in \cite{E}, the existence of nontrivial solutions of the cubic Fermat equation in $K$ is proved based on the discriminant $d_K$ and the class number $h(K)$. In this article, we offer more elementarily computable conditions to determine the existence of nontrivial solutions for the case $k=1$.

\begin{corollary}\label{cor:paperi}
Let $d$ be a squarefree integer and $y^2=x^3-432d^3$ be defined over $\Q$. Assume the Birch and Swinnerton-Dyer conjecture is true when the rank of the elliptic curve is greater than $1$. Then $x^3+y^3=z^3$ has nontrivial solutions over $K$ when $|d|\equiv -1,2,-4,6\pmod 9$.
\end{corollary}

\section{Nontrivial solutions and nontorsion points}\label{sec:proof}

Before we prove Theorem \ref{teo:general}, we will describe all possible trivial solutions to $x^3+y^3=kz^3$ over $K$.

Observe for $x=0$ and $z\neq 0$ that $k=(y/z)^3$, but $k$ is a cubefree integer, so $y=\zeta_3^n z$ and $k=1$. From here, $n=0$ for all $d$ different from $-3$ and $n\in\lbrace 0,1,2\rbrace$ for $d=-3$. Analogously for $y=0$ and $z\neq 0$. Thus $(0,z,z)$ and $(z,0,z)$ are two trivial solutions for $x^3+y^3=z^3$ over $\Q(\sqrt{d})$. Also $(0,\zeta_3^n z,z)$ and $(\zeta_3^n z,0,z)$ are six trivial solutions for $x^3+y^3=z^3$ over $\Q(\sqrt{-3})$.
Also observe, for $z=0$ and $x\neq 0$ we get $x=-y$. So $(x,-x,0)$ is a trivial solution such that $x_0+y_0=0$. Clearly $(0,0,0)$ also is a trivial solution. Now we proceed to prove Theorem \ref{teo:general}:
\begin{proof}
Let $(x_0,y_0)$ be a nontorsion point on $y^2=x^3-432d^3k^2$ over $\Q$, then
\[
\left(\dfrac{y_0}{d^2}\sqrt{d}\right)^2=\left(\dfrac{x_0}{d}\right)^3-432k^2
\]

In this way $(x_1,y_1)=\left(\dfrac{x_0}{d},\dfrac{y_0}{d^2}\sqrt{d}\right)$ is a point on $y^2=x^3-432k^2$ over $K$. 

Thus
\begin{align*}
y_1^2+2^43^3k^2&=x_1^3\\
2^{3}3^{3}y_1^{2}k+2^{7}3^{6}k^3&=2^33^3x_1^3k\\
\left(y_1+36k\right)^{3}+\left(36k-y_1\right)^{3}&=\left(6x_1\right)^3k
\end{align*}
So $(x_2,y_2,z_2)=\left(\dfrac{y_0\sqrt{d}+36kd^2}{d^2},\dfrac{-y_0\sqrt{d}+36kd^2}{d^2},\dfrac{6x_0}{d}\right)$ is a solution to $x^3+y^3=kz^3$ over $K$. Now, suppose $(x_2,y_2,z_2)$ is a trivial solution. Thus
\begin{itemize}
\item When $z_2=0$, then $x_0=0$ and $y_0^2=-2^43^3d^3k^2$, so $y_0=2^23k\sqrt{-3d}$. This is a contradiction unless $d=-3$, in this case $y_0=\pm 36k$, therefore $(x_0,y_0)=(0,\pm 36)$ which is a 3-order $\Q$-point.
\item When $x_2=0$, then $y_0\sqrt{d}=-36kd^2$, so $\sqrt{d}=\dfrac{-36kd^2}{y_0}\in\Q$ which is a contradiction unless $d=1$ and $y_0=-36kd^2$. In this case $72^3k^2=6^3x_0^3$, so $x_0=12\sqrt[3]{k}$, which is a contradiction unless $k=1$. Therefore $(x_0,y_0)=(12,-36)$ which is a 3-order $\Q$-point.
\item When $y_2=0$, we proceed similarly as when $x_2=0$.
\end{itemize}
Therefore $(x_2,y_2,z_2)$ is nontrivial.
\end{proof}

\section{Rank of the elliptic curve $y^2=x^3-432d^3k^2$ over $\Q$}\label{sec:crit}

Let $E_D$ be the elliptic curve defined by
$$E_D\colon y^2=x^3+D$$
where $D$ is an integer. Since many authors have already studied this curve, we have plenty of tools to describe the behaviour of $E_D$. For example, from page 74 of \cite{C}, we have 

\begin{proposition}\label{prop:cambcurv}
Let $D$ be a nonsixth-power free integer, $N$ be an integer such that $D'=\frac{D}{N^{6}}$ is a sixth-power free integer and $\phi_N$ be the isogeny defined by
\begin{equation*}
\phi_N\colon (x,y)\mapsto \left( \frac{x}{N^{2}},\frac{y}{N^{3}}\right)
\end{equation*}
Then $\phi_N$ is an isomorphism between the elliptic curves $E_D$ and $E_{D'}$.
\end{proposition}

As it is known, due to the complex multiplication of the curve $E_D$, the analytic continuation of the $L$-function related to $E_D$, $\Lambda(s)$ fulfils the functional equation $$\Lambda(E_D,s)=w_D\Lambda(E_D,2-s)$$

Observe $\Lambda(E_D,1)=0$, when $w_D=-1$, and in order to determine when $w_D=-1$ we use the following theorem:
\begin{theorem}[\cite{D}]\label{Liv}
Let $D$ be a sixth-power free integer with factorization $D=2^aD_2=3^bD_3$ where $2\nmid D_2$ and $3 \nmid D_3.$ Then, the sign of the functional equation of the elliptic curve $y^2=x^3+D$ is
$$w_D=-w_2 w_3\prod_{\substack{p\vert D\\
                  p\neq 2,3}}w_p,$$
where
\begin{align*}
w_2&=\begin{cases}
-1 & \mbox{when }2\nmid a\\
-1 & \mbox{when }2\mid a,a\neq 4\mbox{ and } D_2\equiv1\pmod4\\
+1 & \mbox{otherwise}
\end{cases}
\end{align*}
\begin{align*}
w_3&=\begin{cases}
-1 & \mbox{when }b\equiv -1\pmod3\\
-1&\mbox{when }3\mid b\mbox{ and } D_3\equiv \pm 2,(-1)^{b+1} \pmod9\\
+1 & \mbox{otherwise}
\end{cases}\\
w_p&=\begin{cases}
-1 & \mbox{when } p\equiv 2 \pmod3,\ p\mbox{ odd prime}\\
+1 & \mbox{otherwise}
\end{cases}
\end{align*}
\end{theorem}

In our case, $D=-432d^3k^2$ with $d$ a squarefree integer and $k$ a cubefree integer. We define $d=2^u3^m d_6$ and $k=2^v3^n k_6$, where $\gcd(6,d_6)=\gcd(6,k_6)=1$, $u,m\in\lbrace 0,1\rbrace$ and $v,n\in\lbrace 0,1,2\rbrace$. It follows that $$D=-2^{4+3u+2v}3^{3+3m+2n}d_6^3k_6^2$$
So $a=4+3u+2v$, $b=3+3m+2n$, $D_2=-3^{3+3m+2n}d_6^3k_6^2$ and $D_3=-2^{4+3u+2v}d_6^3k_6^2$.

Observe $d=2^u 3^m d_6$ implies we have four options for $d$ depending on $u$ and $m$, also $k=2^v 3^n k_6$ where we have nine options for $k$ depending on $v$ and $n$. Thus, we have a total of $36$ cases to describe. Furthermore, $a$ and $b$ can be greater than $6$, however due to Proposition \ref{prop:cambcurv} there exists $E_{D'}$ such that $D'$ is sixth-powerfree, and $E_{D'}$ is isomorphic to $E_{-432d^3k^2}$. Then, we can explore any of the two models to describe its rank.

On the other hand, the congruences $D_2\equiv 1\pmod4$ and $D_3\equiv \pm 2, (-1)^{b+1}\pmod9$ hold for both models $E_{-432d^3k^2}$ and $E_{D'}$. We proceed considering $E_{-432d^3k^2}(\Q)$.

\begin{definition}
Let $signature(E_{-432d^3k^2}(\Q))$ be defined as $(u,m,v,n)$ where $d=d=2^u 3^m d_6$ and $k=2^v 3^n k_6$, $\gcd(6,d_6)=\gcd(6,k_6)=1$.
\end{definition}

Applying Theorem \ref{Liv} to $E_{-432d^3k^2}(\Q)$, we have that
\begin{enumerate}
\item When $a$ is odd, then $w_2=-1$. Since $4+3u+2v\equiv u\pmod2$, we have $a$ is odd when $u=1$. Meaning the signature $(1,\cdot,\cdot,\cdot)$ implies $w_2=-1$.
\item When $a=4$, then $w_2=1$. We have $4=4+3u+2v$, implying $0=3u+2v$. Thus $a=4$, when $u=v=0$. Meaning the signature $(0,\cdot,0,\cdot)$ implies $w_2=1$.
\item When $b\equiv -1\pmod3$, then $w_3=-1$. Since $3+3m+2n\equiv -n\pmod3$, thus $b\equiv -1\pmod3$ when $n=1$. Meaning the signature $(\cdot,\cdot,\cdot,1)$ implies $w_3=-1$.
\item When $b\equiv 1\pmod3$, then $w_3=1$. Since $3+3m+2n\equiv -n\pmod3$, thus $b\equiv 1\pmod3$, when $n=2$. Meaning the signature $(\cdot,\cdot,\cdot,2)$ implies $w_3=1$.
\end{enumerate}

Now, we can explore the total cases following a classification based on the previous observations.

\begin{enumerate}[$i)$]
\item When $a$ is odd or $a=4$, and $b\equiv \pm1\pmod 3$. In this first case, the four possible combinations
\begin{enumerate}
\item When $a$ is odd and $b\equiv 1\pmod3$. Since $a$ is odd, then $u=1$ and $w_2=-1$, also since $b\equiv 1\pmod3$, then $n=2$ and $w_3=1$. Thus $(1,\cdot,\cdot,2)$ implies
$$w_D=-w_2w_3\prod w_p=-(-1)(1)\prod w_p=\prod w_p$$
\item When $a$ is odd and $b\equiv -1\pmod3$. Since $a$ is odd, then $u=1$ and $w_2=-1$, also since $b\equiv -1\pmod3$, then $n=1$ and $w_3=-1$. 

Thus $(1,\cdot,\cdot,1)$ implies
$$w_D=-w_2w_3\prod w_p=-(-1)(-1)\prod w_p=-\prod w_p$$
\item When $a=4$ and $b\equiv 1\pmod3$. Since $a=4$, then $u=v=0$ and $w_2=1$, also since $b\equiv 1\pmod3$, then $n=2$ and $w_3=1$. Thus $(0,\cdot,0,2)$ implies
$$w_D=-w_2w_3\prod w_p=-(1)(1)\prod w_p=-\prod w_p$$
\item When $a=4$ and $b\equiv -1\pmod3$. Since $a=4$, then $u=v=0$ and $w_2=1$, also since $b\equiv -1\pmod3$, then $n=1$ and $w_3=-1$. Thus $(0,\cdot,0,1)$ implies
$$w_D=-w_2w_3\prod w_p=-(1)(-1)\prod w_p=\prod w_p$$
\end{enumerate}

Therefore, we can group $w_D$ into two types of signatures
\begin{equation}\label{Sec4_1}
w_D=\begin{cases}
\ \ \prod w_p&\mbox{, when }(1,\cdot,\cdot,2)\mbox{ or }(0,\cdot,0,1)\\
-\prod w_p&\mbox{, when }(1,\cdot,\cdot,1)\mbox{ or }(0,\cdot,0,2)
\end{cases}
\end{equation}

\item When $a$ is even different from $4$ and $b\equiv \pm 1\pmod3$. Since $a$ is even different from $4$. Then $u=0$ and $v=1,2$ ($v\neq 0$). Also, we know that $$D_2=-3^b d_6^3k_6^2\equiv -3^b d_6\equiv -(-1)^bd_6\pmod4$$
Regarding $b$, since $b\equiv \pm 1\pmod3$ we have $n=1,2$ ($n\neq 0$) and $b=3+3m+2n$, then we have four different values for $b$: $5,7,8,10$. We analyse all cases:
\begin{enumerate}
\item When $b$ is even, then $D_2\equiv -d_6\pmod4$. Since $d_6\equiv \pm 1\pmod4$, then $d_6\equiv w_2\pmod4$. So we have two options:
\begin{itemize}
\item When $(0,1,v\neq 0,2)$, then $b=10\equiv 1\pmod3$, so $w_3=1$ and $w_D=-w_2\prod w_p$, thus
\begin{align*}
w_D=&\begin{cases}
\prod w_p&\mbox{, when }d_6\equiv -1\pmod4\\
-\prod w_p&\mbox{, when }d_6\equiv 1\pmod4
\end{cases}
\end{align*}
\item When $(0,1,v\neq 0,1)$, then $b=8\equiv -1\pmod3$, so $w_3=-1$ and $w_D=w_2\prod w_p$, thus
\begin{align*}
w_D=&\begin{cases}
-\prod w_p&\mbox{, when }d_6\equiv -1\pmod4\\
\prod w_p&\mbox{, when }d_6\equiv 1\pmod4
\end{cases}
\end{align*}
\end{itemize}
\item When $b$ is odd, then $D_2\equiv d_6\pmod4$. Since $d_6\equiv \pm 1\pmod4$, then $d_6\equiv -w_2\pmod4$. So we have two options:
\begin{itemize}
\item When $(0,0,v\neq 0,2)$, then $b=7\equiv 1\pmod3$, so $w_3=1$ and $w_D=-w_2\prod w_p$, thus
\begin{align*}
w_D=&\begin{cases}
-\prod w_p&\mbox{, when }d_6\equiv -1\pmod4\\
\prod w_p&\mbox{, when }d_6\equiv 1\pmod4
\end{cases}
\end{align*}

\item When $(0,0,v\neq 0,1)$, then $b=5\equiv -1\pmod3$, so $w_3=-1$ and $w_D=w_2\prod w_p$, thus
\begin{align*}
w_D=&\begin{cases}
\prod w_p&\mbox{, when }d_6\equiv -1\pmod4\\
-\prod w_p&\mbox{, when }d_6\equiv 1\pmod4
\end{cases}
\end{align*}
\end{itemize}
\end{enumerate}

In this way, we can summarise the signatures $(0,\cdot,v\neq 0,n\neq 0)$ as
\begin{equation}\label{Sec4_2}
w_D=\begin{cases}
-(-1)^{m+n}\prod w_p&\mbox{, when }d_6\equiv -1\pmod4\\
(-1)^{m+n}\prod w_p&\mbox{, when }d_6\equiv 1\pmod4
\end{cases}
\end{equation}

\item When $a$ is odd or $a=4$, and $3\vert b$. For $a$ odd, we know $u=1$ and $w_2=-1$ and for $a=4$ we know $u=v=0$ and $w_2=1$. Meanwhile $b=3+3m+2n$, since $3\vert b$, then $n=0$ and $m=0,1$. However $D_3=-2^ad_6^3k_6^2$, thus
\[D_3=-2^ad_6^3k_6^2\equiv -2^a(\pm 1)k_6^2\mod9\]
where $\pm1$ depends on the congruence $d_6\equiv \pm 1\pmod3$.\\
We have the following cases:
\begin{enumerate}

\item When $a=4$ and $b=3$, the signature is $(0,0,0,0)$. Since $a=4$, then $w_2=1$. On the other hand
$$D_3\equiv -2^4(\pm 1)k_6^2\equiv (\pm 1)2k_6^2\pmod9$$
We know $\gcd(k_6,6)=1$, then $k_6\equiv \pm 1,\pm 2,\pm 4\pmod9$. Furthermore, $k_6^2\equiv 1,4,-2\pmod9$ and $2k_6^2\equiv 2,-1,-4\pmod9$ respectively. 

Then 
\begin{equation*}
\resizebox{.9\hsize}{!}{$
D_3\equiv\begin{cases}
+2\pmod9&\mbox{, when }d_6\equiv 1\pmod3\mbox{ and }k_6\equiv\pm1\pmod9\\
-2\pmod9&\mbox{, when }d_6\equiv -1\pmod3\mbox{ and }k_6\equiv\pm1\pmod9\\
-1\pmod9&\mbox{, when }d_6\equiv 1\pmod3\mbox{ and }k_6\equiv\pm2\pmod9\\
+1\pmod9&\mbox{, when }d_6\equiv -1\pmod3\mbox{ and }k_6\equiv\pm2\pmod9\\
-4\pmod9&\mbox{, when }d_6\equiv 1\pmod3\mbox{ and }k_6\equiv\pm4\pmod9\\
+4\pmod9&\mbox{, when }d_6\equiv -1\pmod3\mbox{ and }k_6\equiv\pm4\pmod9
\end{cases}$}
\end{equation*}
Applying Theorem \ref{Liv} to $D=3^bD_3$, we get that
\begin{equation*}
\resizebox{.65\hsize}{!}{$
w_3=\begin{cases}
-1&\mbox{, when }3\vert b\mbox{ and }D_3\equiv \pm2,1\pmod 9\\
+1&\mbox{, when }3\vert b\mbox{ and }D_3\equiv \pm4,-1\pmod 9
\end{cases}$}
\end{equation*}
then
\begin{equation*}
\resizebox{.8\hsize}{!}{$
w_3=\begin{cases}
-1&\mbox{, when }d_6\equiv 1\pmod3\mbox{ and }k_6\equiv\pm1\pmod9\\
-1&\mbox{, when }d_6\equiv -1\pmod3\mbox{ and }k_6\equiv\pm1\pmod9\\
+1&\mbox{, when }d_6\equiv 1\pmod3\mbox{ and }k_6\equiv\pm2\pmod9\\
-1&\mbox{, when }d_6\equiv -1\pmod3\mbox{ and }k_6\equiv\pm2\pmod9\\
+1&\mbox{, when }d_6\equiv 1\pmod3\mbox{ and }k_6\equiv\pm4\pmod9\\
+1&\mbox{, when }d_6\equiv -1\pmod3\mbox{ and }k_6\equiv\pm4\pmod9
\end{cases}$}
\end{equation*}

since $w_D=-w_3\prod w_p$, thus

\begin{equation*}
\resizebox{.8\hsize}{!}{$
w_D=\begin{cases}
+\prod w_p&\mbox{, when }d_6\equiv 1\pmod3\mbox{ and }k_6\equiv\pm1\pmod9\\
+\prod w_p&\mbox{, when }d_6\equiv -1\pmod3\mbox{ and }k_6\equiv\pm1\pmod9\\
-\prod w_p&\mbox{, when }d_6\equiv 1\pmod3\mbox{ and }k_6\equiv\pm2\pmod9\\
+\prod w_p&\mbox{, when }d_6\equiv -1\pmod3\mbox{ and }k_6\equiv\pm2\pmod9\\
-\prod w_p&\mbox{, when }d_6\equiv 1\pmod3\mbox{ and }k_6\equiv\pm4\pmod9\\
-\prod w_p&\mbox{, when }d_6\equiv -1\pmod3\mbox{ and }k_6\equiv\pm4\pmod9
\end{cases}$}
\end{equation*}
Therefore, for $a=4$ and $b=3$ (signature $(0,0,0,0)$)
\begin{equation*}
\resizebox{.8\hsize}{!}{$
w_D=\begin{cases}
+\prod w_p&\mbox{, when }d_6\equiv \pm1\pmod3\mbox{ and }k_6\equiv\pm1\pmod9\\
+\prod w_p&\mbox{, when }d_6\equiv -1\pmod3\mbox{ and }k_6\equiv\pm2\pmod9\\
-\prod w_p&\mbox{, when }d_6\equiv 1\pmod3\mbox{ and }k_6\equiv\pm2\pmod9\\
-\prod w_p&\mbox{, when }d_6\equiv \pm1\pmod3\mbox{ and }k_6\equiv\pm4\pmod9
\end{cases}$}
\end{equation*}
\item Replicating the analysis done in $(a)$ for $a=4$ and $b=6$ (signature $(0,1,0,0)$), we conclude:
\begin{equation*}
\resizebox{.8\hsize}{!}{$
w_D=\begin{cases}
+\prod w_p&\mbox{, when }d_6\equiv \pm1\pmod3\mbox{ and }k_6\equiv\pm1\pmod9\\
+\prod w_p&\mbox{, when }d_6\equiv 1\pmod3\mbox{ and }k_6\equiv\pm2\pmod9\\
-\prod w_p&\mbox{, when }d_6\equiv -1\pmod3\mbox{ and }k_6\equiv\pm2\pmod9\\
-\prod w_p&\mbox{, when }d_6\equiv \pm1\pmod3\mbox{ and }k_6\equiv\pm4\pmod9
\end{cases}$}
\end{equation*}

\item Replicating the analysis done in $(a)$ for $a=7$ and $b=3$ (signature is $(1,0,0,0)$), we conclude:
\begin{equation*}
\resizebox{.8\hsize}{!}{$
w_D=\begin{cases}
+\prod w_p&\mbox{, when }d_6\equiv \pm1\pmod3\mbox{ and }k_6\equiv\pm4\pmod9\\
+\prod w_p&\mbox{, when }d_6\equiv -1\pmod3\mbox{ and }k_6\equiv\pm2\pmod9\\
-\prod w_p&\mbox{, when }d_6\equiv 1\pmod3\mbox{ and }k_6\equiv\pm2\pmod9\\
-\prod w_p&\mbox{, when }d_6\equiv \pm1\pmod3\mbox{ and }k_6\equiv\pm1\pmod9
\end{cases}$}
\end{equation*}

\item Replicating the analysis done in $(a)$ for $a=7$ and $b=6$ (signature $(1,1,0,0)$), lead us to:
\begin{equation*}
\resizebox{.8\hsize}{!}{$
w_D=\begin{cases}
+\prod w_p&\mbox{, when }d_6\equiv \pm1\pmod3\mbox{ and }k_6\equiv\pm4\pmod9\\
+\prod w_p&\mbox{, when }d_6\equiv 1\pmod3\mbox{ and }k_6\equiv\pm2\pmod9\\
-\prod w_p&\mbox{, when }d_6\equiv -1\pmod3\mbox{ and }k_6\equiv\pm2\pmod9\\
-\prod w_p&\mbox{, when }d_6\equiv \pm1\pmod3\mbox{ and }k_6\equiv\pm1\pmod9
\end{cases}$}
\end{equation*}

\item Replicating the analysis done in $(a)$ for $a=9$ and $b=3$ (signature is $(1,0,1,0)$), we get:
\begin{equation*}
\resizebox{.8\hsize}{!}{$
w_D=\begin{cases}
+\prod w_p&\mbox{when }d_6\equiv \pm1\pmod3\mbox{ and }k_6\equiv\pm2\pmod9\\
+\prod w_p&\mbox{when }d_6\equiv -1\pmod3\mbox{ and }k_6\equiv\pm1\pmod9\\
-\prod w_p&\mbox{when }d_6\equiv 1\pmod3\mbox{ and }k_6\equiv\pm1\pmod9\\
-\prod w_p&\mbox{when }d_6\equiv \pm1\pmod3\mbox{ and }k_6\equiv\pm4\pmod9
\end{cases}$}
\end{equation*}

\item Replicating the analysis done in $(a)$ for $a=9$ and $b=6$, here the signature is $(1,1,1,0)$), we conclude:
\begin{equation*}
\resizebox{.8\hsize}{!}{$
w_D=\begin{cases}
+\prod w_p&\mbox{when }d_6\equiv \pm1\pmod3\mbox{ and }k_6\equiv\pm2\pmod9\\
+\prod w_p&\mbox{when }d_6\equiv 1\pmod3\mbox{ and }k_6\equiv\pm1\pmod9\\
-\prod w_p&\mbox{when }d_6\equiv -1\pmod3\mbox{ and }k_6\equiv\pm1\pmod9\\
-\prod w_p&\mbox{when }d_6\equiv \pm1\pmod3\mbox{ and }k_6\equiv\pm4\pmod9
\end{cases}$}
\end{equation*}

\item Replicating the analysis done in $(a)$ for $a=11$ and $b=3$ (signature is $(1,0,2,0)$), we conclude:
\begin{equation*}
\resizebox{.8\hsize}{!}{$
w_D=\begin{cases}
+\prod w_p&\mbox{when }d_6\equiv \pm1\pmod3\mbox{ and }k_6\equiv\pm1\pmod9\\
+\prod w_p&\mbox{when }d_6\equiv -1\pmod3\mbox{ and }k_6\equiv\pm4\pmod9\\
-\prod w_p&\mbox{when }d_6\equiv 1\pmod3\mbox{ and }k_6\equiv\pm4\pmod9\\
-\prod w_p&\mbox{when }d_6\equiv \pm1\pmod3\mbox{ and }k_6\equiv\pm2\pmod9
\end{cases}$}
\end{equation*}

\item Replicating the analysis done in $(a)$ for $a=11$ and $b=6$ (signature is $(1,1,2,0)$), we get:
\begin{equation*}
\resizebox{.8\hsize}{!}{$
w_D=\begin{cases}
+\prod w_p&\mbox{when }d_6\equiv \pm1\pmod3\mbox{ and }k_6\equiv\pm1\pmod9\\
+\prod w_p&\mbox{when }d_6\equiv 1\pmod3\mbox{ and }k_6\equiv\pm4\pmod9\\
-\prod w_p&\mbox{when }d_6\equiv -1\pmod3\mbox{ and }k_6\equiv\pm4\pmod9\\
-\prod w_p&\mbox{when }d_6\equiv \pm1\pmod3\mbox{ and }k_6\equiv\pm2\pmod9
\end{cases}$}
\end{equation*}
\end{enumerate}

In this way, for the signature $(\cdot, \cdot,0,0)$, we have that
\begin{equation}\label{Sec4_3}
\resizebox{.8\hsize}{!}{$
w_D=\begin{cases}
(-1)^u\prod w_p&\mbox{when }d_6\equiv \pm 1\pmod3\mbox{ and }k_6\equiv \pm 1\pmod9\\
(-1)^{u+1}\prod w_p&\mbox{when }d_6\equiv \pm 1\pmod3\mbox{ and }k_6\equiv \pm 4\pmod9\\
(-1)^{m+1}\prod w_p&\mbox{when }d_6\equiv 1\pmod3\mbox{ and }k_6\equiv \pm 2\pmod9\\
(-1)^m\prod w_p&\mbox{when }d_6\equiv -1\pmod3\mbox{ and }k_6\equiv \pm 2\pmod9
\end{cases}$}
\end{equation}

Finally, for the signature $(1, \cdot,v\neq 0,0)$ we have that
\begin{equation}\label{Sec4_4}
\resizebox{.9\hsize}{!}{$
w_D=\begin{cases}
+\prod w_p&\mbox{when }d_6\equiv \pm 1\pmod3\mbox{ and }k_6\equiv \pm 2\cdot 4^{v-1}\pmod9\\
+\prod w_p&\mbox{when }d_6\equiv (-1)^{m+1}\pmod3\mbox{ and }k_6\equiv \pm 1\cdot 4^{v-1}\pmod9\\
-\prod w_p&\mbox{when }d_6\equiv (-1)^m\pmod3\mbox{ and }k_6\equiv \pm 1\cdot 4^{v-1}\pmod9\\
-\prod w_p&\mbox{when }d_6\equiv \pm 1\pmod3\mbox{ and }k_6\equiv \pm 4\cdot 4^{v-1}\pmod9
\end{cases}$}
\end{equation}

\item When $a\neq 4$ and $3\vert b$, then $u=0$ and $v=1,2$ $(v\neq 0)$. Also, since $3\vert b$, then $n=0$ and $m=0,1$. So, we have four different signatures: $(0,0,1,0)$, $(0,0,2,0)$, $(0,1,1,0)$ and $(0,1,2,0)$. On the other hand,
\begin{align*}
D_2&=-3^bd_6^3k_6^2\equiv -(-1)^bd_6\pmod4\\
D_3&=-2^ad_6^3k_6^2\equiv -2^a(\pm 1)k_6^2\pmod9
\end{align*}
where $(\pm 1)$ depends on the congruence of $d_6$ module $3$. Now
\begin{enumerate}
\item When $(0,0,1,0)$, for $a=4+3u+2v=6$ and $b=3+3m+2n=3$
\begin{align*}
D_2&\equiv d_6\pmod4\\
D_3&\equiv -(\pm 1)k_6^2\pmod9
\end{align*}
where $(\pm 1)$ depends on the congruence of $d_6$ on module $3$.
By Theorem \ref{Liv}, we know that
\begin{equation*}
\resizebox{.5\hsize}{!}{$
w_2=\begin{cases}
-1&\mbox{, when }d_6\equiv \phantom{-}1\pmod4\\
+1&\mbox{, when }d_6\equiv -1\pmod4
\end{cases}$}
\end{equation*}
On the other hand, $k_6\equiv \pm1,\pm2,\pm4\pmod9$ since $\gcd(k_6,6)=1$, so $k_6^2\equiv 1,4,-2\pmod9$ respectively, moreover $-k_6^2\equiv -1,-4,2\pmod9$, thus
\begin{equation*}
\resizebox{.8\hsize}{!}{$
D_3=\begin{cases}
-1\pmod9&\mbox{when }k_6\equiv \pm 1\pmod9\mbox{ and }d_6\equiv 1\pmod3\\
\phantom{-}1\pmod9&\mbox{when }k_6\equiv \pm 1\pmod9\mbox{ and }d_6\equiv -1\pmod3\\
-4\pmod9&\mbox{when }k_6\equiv \pm 2\pmod9\mbox{ and }d_6\equiv 1\pmod3\\
\phantom{-}4\pmod9&\mbox{when }k_6\equiv \pm 2\pmod9\mbox{ and }d_6\equiv -1\pmod3\\
\phantom{-}2\pmod9&\mbox{when }k_6\equiv \pm 4\pmod9\mbox{ and }d_6\equiv 1\pmod3\\
-2\pmod9&\mbox{when }k_6\equiv \pm 4\pmod9\mbox{ and }d_6\equiv -1\pmod3
\end{cases}$}
\end{equation*}
Applying Theorem \ref{Liv} to $D=3^3D_3$, we get
\begin{equation*}
\resizebox{.6\hsize}{!}{$
w_3=\begin{cases}
-1&\mbox{when }3\vert b\mbox{ and }D_3\equiv \pm2, 1\pmod9\\
\phantom{-}1&\mbox{when }3\vert b\mbox{ and }D_3\equiv \pm4, -1\pmod9
\end{cases}$}
\end{equation*}
thus
\begin{equation*}
\resizebox{.7\hsize}{!}{$
w_3=\begin{cases}
\phantom{-}1&\mbox{when }k_6\equiv \pm 1\pmod9\mbox{ and }d_6\equiv 1\pmod3\\
-1&\mbox{when }k_6\equiv \pm 1\pmod9\mbox{ and }d_6\equiv -1\pmod3\\
\phantom{-}1&\mbox{when }k_6\equiv \pm 2\pmod9\mbox{ and }d_6\equiv 1\pmod3\\
\phantom{-}1&\mbox{when }k_6\equiv \pm 2\pmod9\mbox{ and }d_6\equiv -1\pmod3\\
-1&\mbox{when }k_6\equiv \pm 4\pmod9\mbox{ and }d_6\equiv 1\pmod3\\
-1&\mbox{when }k_6\equiv \pm 4\pmod9\mbox{ and }d_6\equiv -1\pmod3
\end{cases}$}
\end{equation*}
Observe that
\begin{equation*}
\resizebox{.8\hsize}{!}{$
w_2w_3=\begin{cases}
-1&\mbox{when }k_6\equiv \pm 1\pmod9\mbox{ and }d_6\equiv \pm 1\pmod{12}\\
\phantom{-}1&\mbox{when }k_6\equiv \pm 1\pmod9\mbox{ and }d_6\equiv \pm 5\pmod{12}\\
-1&\mbox{when }k_6\equiv \pm 2\pmod9\mbox{ and }d_6\equiv 1,5\pmod{12}\\
\phantom{-}1&\mbox{when }k_6\equiv \pm 2\pmod9\mbox{ and }d_6\equiv -1,-5\pmod{12}\\
\phantom{-}1&\mbox{when }k_6\equiv \pm 4\pmod9\mbox{ and }d_6\equiv 1,5\pmod{12}\\
-1&\mbox{when }k_6\equiv \pm 4\pmod9\mbox{ and }d_6\equiv -1,-5\pmod{12}
\end{cases}$}
\end{equation*}
It is important to remark that $w_D=\pm\prod w_p$ depending on $w_2w_3=\mp1$.  Therefore, we have the following cases:
\begin{equation*}
\resizebox{.8\hsize}{!}{$
w_D=\begin{cases}
-\prod w_p&\mbox{when }k_6\equiv \pm 1\pmod9\mbox{ and }d_6\equiv \pm 5\pmod{12}\\
-\prod w_p&\mbox{when }k_6\equiv \pm 2\pmod9\mbox{ and }d_6\equiv -1,-5\pmod{12}\\
-\prod w_p&\mbox{when }k_6\equiv \pm 4\pmod9\mbox{ and }d_6\equiv 1,5\pmod{12}\\
\prod w_p&\mbox{when }k_6\equiv \pm 1\pmod9\mbox{ and }d_6\equiv \pm 1\pmod{12}\\
\prod w_p&\mbox{when }k_6\equiv \pm 2\pmod9\mbox{ and }d_6\equiv 1,5\pmod{12}\\
\prod w_p&\mbox{when }k_6\equiv \pm 4\pmod9\mbox{ and }d_6\equiv -1,-5\pmod{12}
\end{cases}$}
\end{equation*}
\item Replicating the analysis done in $(a)$ for the signature $(0,0,2,0)$ lead us to the following cases:
\begin{equation*}
\resizebox{.8\hsize}{!}{$
w_D=\begin{cases}
-\prod w_p&\mbox{when }k_6\equiv \pm 1\pmod9\mbox{ and }d_6\equiv -1,-5\pmod{12}\\
-\prod w_p&\mbox{when }k_6\equiv \pm 2\pmod9\mbox{ and }d_6\equiv 1,5\pmod{12}\\
-\prod w_p&\mbox{when }k_6\equiv \pm 4\pmod9\mbox{ and }d_6\equiv \pm 5\pmod{12}\\
\prod w_p&\mbox{when }k_6\equiv \pm 1\pmod9\mbox{ and }d_6\equiv 1,5\pmod{12}\\
\prod w_p&\mbox{when }k_6\equiv \pm 2\pmod9\mbox{ and }d_6\equiv -1,-5\pmod{12}\\
\prod w_p&\mbox{when }k_6\equiv \pm 4\pmod9\mbox{ and }d_6\equiv \pm 1\pmod{12}
\end{cases}$}
\end{equation*}
\item Replicating the analysis done in $(a)$ for the signature $(0,1,1,0)$ we get the following cases:
\begin{equation*}
\resizebox{.8\hsize}{!}{$
w_D=\begin{cases}
-\prod w_p&\mbox{when }k_6\equiv \pm 1\pmod9\mbox{ and }d_6\equiv \pm 5\pmod{12}\\
-\prod w_p&\mbox{when }k_6\equiv \pm 2\pmod9\mbox{ and }d_6\equiv 1,5\pmod{12}\\
-\prod w_p&\mbox{when }k_6\equiv \pm 4\pmod9\mbox{ and }d_6\equiv -1,-5\pmod{12}\\
\prod w_p&\mbox{when }k_6\equiv \pm 1\pmod9\mbox{ and }d_6\equiv \pm 1\pmod{12}\\
\prod w_p&\mbox{when }k_6\equiv \pm 2\pmod9\mbox{ and }d_6\equiv -1,-5\pmod{12}\\
\prod w_p&\mbox{when }k_6\equiv \pm 4\pmod9\mbox{ and }d_6\equiv 1,5\pmod{12}
\end{cases}$}
\end{equation*}
\item Finally when signature is $(0,1,2,0)$ we get the following cases:
\begin{equation*}
\resizebox{.8\hsize}{!}{$
w_D=\begin{cases}
-\prod w_p&\mbox{when }k_6\equiv \pm 1\pmod9\mbox{ and }d_6\equiv 1, 5\pmod{12}\\
-\prod w_p&\mbox{when }k_6\equiv \pm 2\pmod9\mbox{ and }d_6\equiv -1,-5\pmod{12}\\
-\prod w_p&\mbox{when }k_6\equiv \pm 4\pmod9\mbox{ and }d_6\equiv \pm 5\pmod{12}\\
\prod w_p&\mbox{when }k_6\equiv \pm 1\pmod9\mbox{ and }d_6\equiv -1,-5\pmod{12}\\
\prod w_p&\mbox{when }k_6\equiv \pm 2\pmod9\mbox{ and }d_6\equiv 1,5\pmod{12}\\
\prod w_p&\mbox{when }k_6\equiv \pm 4\pmod9\mbox{ and }d_6\equiv \pm 1\pmod{12}
\end{cases}$}
\end{equation*}
\end{enumerate}
In this way, for the signature $(0,\cdot,v\neq 0,0)$, we have that
\begin{equation}\label{Sec4_5}
\resizebox{0.9\hsize}{!}{$w_D=\begin{cases}
-\prod w_p&\mbox{when }k_6\equiv \pm 2^{2v-2}\pmod9\mbox{ and }d_6\equiv \pm 5\pmod{12}\\
(-1)^{m+1}\prod w_p&\mbox{when }k_6\equiv \pm 2^{2v-1}\pmod9\mbox{ and }d_6\equiv -1,-5\pmod{12}\\
(-1)^{m+1}\prod w_p&\mbox{when }k_6\equiv \pm 2^{2v}\pmod9\mbox{ and }d_6\equiv 1, 5\pmod{12}\\
\prod w_p&\mbox{when }k_6\equiv \pm 2^{2v-2}\pmod9\mbox{ and }d_6\equiv \pm 1\pmod{12}\\
(-1)^{m}\prod w_p&\mbox{when }k_6\equiv \pm 2^{2v-1}\pmod9\mbox{ and }d_6\equiv 1,5\pmod{12}\\
(-1)^{m}\prod w_p&\mbox{when }k_6\equiv \pm 2^{2v}\pmod9\mbox{ and }d_6\equiv -1,-5\pmod{12}
\end{cases}$}
\end{equation}
\end{enumerate}

\section{A particular case for $\prod w_p$}\label{sec:sum}

Following the notation of Theorem \ref{Liv}, we define $D_6=\frac{D}{2^a3^b}$. According to Theorem \ref{Liv} for every prime $p\vert D_6$, $w_p=-1$ if and only if $p\equiv -1\pmod3$, meaning $w_p\equiv p\pmod3$. Thus
\[
\displaystyle\prod_{p\vert D_6}w_p\equiv\prod_{\substack{p\vert D_6\\
                  p\equiv -1\pmod3}}p\pmod3
\]
In our case $D_6=d_6^3k_6^2$, then $D_6\equiv d_6\pmod3$. As we can notice, this is not enough to relate this congruence to the parity of the number of primes dividing $D_6$ and congruent to $-1$ module $3$ since all the primes dividing $k_6$ vanished from the congruence. 

Observe $k_6=k_{61}k_{62}^2$ where $\gcd(k_{61},k_{62})=1$ and both are squarefree integers. Thus, two possible cases keep all primes in congruence; we have either $\gcd(d_6,k_{61})=1$ and $k_{62}=1$, or $\gcd(d_6,k_{61})=1$ and $\gcd(d_6,k_{62})=k_{62}$. In both cases, the following congruence holds:
\begin{align}\label{eqn:cong}
\prod_{p\vert D_6}w_p&\equiv |d_6|k_6\pmod3
\end{align}
From now on, we assume $\gcd(d_6,k_{61})=1$ and $k_{62}=1$. 

\begin{theorem}\label{thm:table}
Let $d$ be a squarefree integer and $k$ a cubefree integer such that $k>0$. Also, let $d=2^u3^md_6$ and $k=2^v3^nk_6$ with $\gcd(6,d_6)=\gcd(6,k_6)=\gcd(d_6,k_6)=1$ and $k_6$ a squarefree integer. Consider $(u,m,v,n)$ the signature of the elliptic curve $y^2=x^3-432d^3k^2$ defined over $\Q$. Assume the Birch and Swinnerton-Dyer conjecture is true when the elliptic curve rank is greater than 1. Then, there exist nontrivial solutions to $x^3+y^3=kz^3$ over $\Q(\sqrt{d})$ according to the following conditions:

\begin{center}
\begin{adjustbox}{height=10 cm, width = 11 cm}
\begin{tabular}{|c|c|}
\hline
Signature & Criteria\\
\hline
$\begin{array}{cc}
(1,0,0,1)&(1,0,1,1)\\
(1,0,2,1)&(1,1,0,1)\\
(1,1,1,1)&(1,1,2,1)\\
(0,0,0,2)&(0,1,0,2)
\end{array}$
& $|d_6|k_6\equiv 1\pmod3$\\
\hline
$\begin{array}{cc}
(1,0,0,2)&(1,0,1,2)\\
(1,0,2,2)&(1,1,0,2)\\
(1,1,1,2)&(1,1,2,2)\\
(0,0,0,1)&(0,1,0,1)
\end{array}$
& $|d_6|k_6\equiv -1\pmod3$\\
\hline
$\begin{array}{cc}
(0,1,1,1)&(0,0,1,2)\\
(0,0,2,2)&(0,1,2,1)
\end{array}$
&
$
\begin{array}{c}
d_6\equiv -1\pmod4\mbox{ and }|d_6|k_6\equiv 1\pmod3, or\\
d_6\equiv 1\pmod4\mbox{ and }|d_6|k_6\equiv -1\pmod3
\end{array}$
\\
\hline
$\begin{array}{cc}
(0,0,1,1)&(0,1,1,2)\\
(0,1,2,2)&(0,0,2,1)
\end{array}$
&
$\begin{array}{c}
d_6\equiv -1\pmod4\mbox{ and }|d_6|k_6\equiv -1\pmod3, or\\
d_6\equiv 1\pmod4\mbox{ and }|d_6|k_6\equiv 1\pmod3\\
\end{array}$\\
\hline
$(0,0,0,0)$&
$\begin{array}{c}
|d_6|\equiv -1\pmod3\mbox{ and }k_6\equiv 1,-4\pmod9, or\\
|d_6|\equiv 1\pmod3\mbox{ and }k_6\equiv -1,4\pmod9, or\\
d_6>0, d\equiv\pm 1\pmod3\mbox{ and }k_6\equiv-2\pmod9, or\\
d_6<0, d\equiv\pm 1\pmod3\mbox{ and }k_6\equiv 2\pmod9\\
\end{array}$\\
\hline
$(0,1,0,0)$&
$\begin{array}{c}
|d_6|\equiv -1\pmod3\mbox{ and }k_6\equiv 1,-4\pmod9, or\\
|d_6|\equiv 1\pmod3\mbox{ and }k_6\equiv -1,4\pmod9, or\\
d_6>0, d\equiv\pm 1\pmod3\mbox{ and }k_6\equiv 2\pmod9, or\\
d_6<0, d\equiv\pm 1\pmod3\mbox{ and }k_6\equiv-2\pmod9\\
\end{array}$\\
\hline
$(1,0,0,0)$&
$\begin{array}{c}
|d_6|\equiv -1\pmod3\mbox{ and }k_6\equiv -1,4\pmod9, or\\
|d_6|\equiv 1\pmod3\mbox{ and }k_6\equiv 1,-4\pmod9, or\\
d_6>0, d\equiv\pm 1\pmod3\mbox{ and }k_6\equiv-2\pmod9, or\\
d_6<0, d\equiv\pm 1\pmod3\mbox{ and }k_6\equiv 2\pmod9\\
\end{array}$\\
\hline
$(1,1,0,0)$&
$\begin{array}{c}
|d_6|\equiv -1\pmod3\mbox{ and }k_6\equiv -1,4\pmod9, or\\
|d_6|\equiv 1\pmod3\mbox{ and }k_6\equiv 1,-4\pmod9, or\\
d_6>0, d\equiv\pm 1\pmod3\mbox{ and }k_6\equiv 2\pmod9, or\\
d_6<0, d\equiv\pm 1\pmod3\mbox{ and }k_6\equiv -2\pmod9\\
\end{array}$\\
\hline
$(1,0,1,0)$&
$\begin{array}{c}
|d_6|\equiv 1\pmod3\mbox{ and }k_6\equiv 2,4\pmod9, or\\
|d_6|\equiv -1\pmod3\mbox{ and }k_6\equiv -2,-4\pmod9, or\\
d_6>0\mbox{ and }k_6\equiv 1\pmod9, or\\
d_6<0\mbox{ and }k_6\equiv -1\pmod9\\
\end{array}$\\
\hline
$(1,0,2,0)$&
$\begin{array}{c}
|d_6|\equiv 1\pmod3\mbox{ and }k_6\equiv -1,-2\pmod9, or\\
|d_6|\equiv -1\pmod3\mbox{ and }k_6\equiv 1,2\pmod9, or\\
d_6>0\mbox{ and }k_6\equiv 4\pmod9, or\\
d_6<0\mbox{ and }k_6\equiv -4\pmod9\\
\end{array}$\\
\hline
$(1,1,1,0)$&
$\begin{array}{c}
|d_6|\equiv 1\pmod3\mbox{ and }k_6\equiv 2,4\pmod9, or\\
|d_6|\equiv -1\pmod3\mbox{ and }k_6\equiv -2,-4\pmod9, or\\
d_6>0\mbox{ and }k_6\equiv -1\pmod9, or\\
d_6<0\mbox{ and }k_6\equiv 1\pmod9\\
\end{array}$\\
\hline
$(1,1,2,0)$&
$\begin{array}{c}
|d_6|\equiv 1\pmod3\mbox{ and }k_6\equiv -1,-2\pmod9, or\\
|d_6|\equiv -1\pmod3\mbox{ and }k_6\equiv 1,2\pmod9, or\\
d_6>0\mbox{ and }k_6\equiv -4\pmod9, or\\
d_6<0\mbox{ and }k_6\equiv 4\pmod9\\
\end{array}$\\
\hline
\end{tabular}
\end{adjustbox}
\end{center}

\begin{center}
\begin{adjustbox}{height = 5 cm, width = 9.5 cm}
\begin{tabular}{|c|c|}
\hline
Signature & Criteria\\
\hline
$(0,0,1,0)$&
$\begin{array}{c}
|d_6|\equiv 1,5\pmod{12}\mbox{ and }k_6\equiv -1\pmod9, or\\
|d_6|\equiv -1,-5\pmod{12}\mbox{ and }k_6\equiv 1\pmod9, or\\
d_6>0,\ d_6\equiv\pm 1\pmod{12}\mbox{ and }k_6\equiv 2,4\pmod9, or\\
d_6>0,\ d_6\equiv\pm 5\pmod{12}\mbox{ and }k_6\equiv -2,-4\pmod9, or\\
d_6<0,\ d_6\equiv\pm 1\pmod{12}\mbox{ and }k_6\equiv 1,2\pmod9, or\\
d_6<0,\ d_6\equiv\pm 5\pmod{12}\mbox{ and }k_6\equiv -1,-2\pmod9\\
\end{array}$\\
\hline
$(0,0,2,0)$&
$\begin{array}{c}
|d_6|\equiv 1,5\pmod{12}\mbox{ and }k_6\equiv -4\pmod9, or\\
|d_6|\equiv -1,-5\pmod{12}\mbox{ and }k_6\equiv 4\pmod9, or\\
d_6>0,\ d_6\equiv\pm 1\pmod{12}\mbox{ and }k_6\equiv -1,-2\pmod9, or\\
d_6>0,\ d_6\equiv\pm 5\pmod{12}\mbox{ and }k_6\equiv 1,2\pmod9, or\\
d_6<0,\ d_6\equiv\pm 1\pmod{12}\mbox{ and }k_6\equiv 1,2\pmod9, or\\
d_6<0,\ d_6\equiv\pm 5\pmod{12}\mbox{ and }k_6\equiv -1,-2\pmod9\\
\end{array}$\\
\hline
$(0,1,1,0)$&
$\begin{array}{c}
|d_6|\equiv 1,5\pmod{12}\mbox{ and }k_6\equiv -1\pmod9, or\\
|d_6|\equiv -1,-5\pmod{12}\mbox{ and }k_6\equiv 1\pmod9, or\\
d_6>0,\ d_6\equiv\pm 1\pmod{12}\mbox{ and }k_6\equiv -2,-4\pmod9, or\\
d_6>0,\ d_6\equiv\pm 5\pmod{12}\mbox{ and }k_6\equiv 2,4\pmod9, or\\
d_6<0,\ d_6\equiv\pm 1\pmod{12}\mbox{ and }k_6\equiv 2,4\pmod9, or\\
d_6<0,\ d_6\equiv\pm 5\pmod{12}\mbox{ and }k_6\equiv -2,-4\pmod9\\
\end{array}$\\
\hline
$(0,1,2,0)$&
$\begin{array}{c}
|d_6|\equiv 1,5\pmod{12}\mbox{ and }k_6\equiv -4\pmod9, or\\
|d_6|\equiv -1,-5\pmod{12}\mbox{ and }k_6\equiv 4\pmod9, or\\
d_6>0,\ d_6\equiv\pm 1\pmod{12}\mbox{ and }k_6\equiv 1,2\pmod9, or\\
d_6>0,\ d_6\equiv\pm 5\pmod{12}\mbox{ and }k_6\equiv -1,-2\pmod9, or\\
d_6<0,\ d_6\equiv\pm 1\pmod{12}\mbox{ and }k_6\equiv -1,-2\pmod9, or\\
d_6<0,\ d_6\equiv\pm 5\pmod{12}\mbox{ and }k_6\equiv 1,2\pmod9\\
\end{array}$\\
\hline
\end{tabular}
\end{adjustbox}
\end{center}
\end{theorem}

\begin{proof}
According to Theorem \ref{teo:general} and considering that for $\Lambda(E_D,1)=0$, a sufficient condition is $w_D=-1$, then it is enough to apply the criteria developed for each group of signatures developed in Section \ref{sec:crit} as follows:

\begin{enumerate}[$i)$]
\item For $(1,\cdot,\cdot,n\neq 0)$ or $(0,\cdot,0,n\neq 0)$ we apply (\ref{eqn:cong}) in (\ref{Sec4_1}) obtaining:
\begin{equation*}
\resizebox{.6\hsize}{!}{$
w_D=\begin{cases}
(-1)^{u+n+1} &\mbox{, when }|d_6|k_6\equiv 1\pmod3\\
(-1)^{u+n} &\mbox{, when }|d_6|k_6\equiv -1\pmod3
\end{cases}$}
\end{equation*}
\item For $(0,\cdot,v\neq 0,n\neq 0)$  we apply (\ref{eqn:cong}) in (\ref{Sec4_2}) obtaining:
\begin{equation*}
\resizebox{.8\hsize}{!}{$
w_D\equiv\begin{cases}
(-1)^{m+n+1}&\mbox{, when }d_6\equiv -1\pmod4\mbox{ and }|d_6|k_6\equiv 1\pmod3\\
(-1)^{m+n}&\mbox{, when }d_6\equiv -1\pmod4\mbox{ and }|d_6|k_6\equiv -1\pmod3\\
(-1)^{m+n}&\mbox{, when }d_6\equiv 1\pmod4\mbox{ and }|d_6|k_6\equiv 1\pmod3\\
(-1)^{m+n+1}&\mbox{, when }d_6\equiv 1\pmod4\mbox{ and }|d_6|k_6\equiv -1\pmod3
\end{cases}$}
\end{equation*}
\item 
\begin{enumerate}
\item For $(\cdot, \cdot,0,0)$ we apply (\ref{eqn:cong}) in (\ref{Sec4_3}) obtaining:
\begin{equation*}
\resizebox{.9\hsize}{!}{$
w_D\equiv\begin{cases}
(-1)^u |d_6|k_6\pmod3&\mbox{when }k_6\equiv \pm 1\pmod9\\
(-1)^{u+1} |d_6|k_6\pmod3&\mbox{when }k_6\equiv \pm 4\pmod9\\
(-1)^{m+1} |d_6|k_6\pmod3&\mbox{when }d_6\equiv 1\pmod3\mbox{ and }k_6\equiv \pm 2\pmod9\\
(-1)^m |d_6|k_6\pmod3&\mbox{when }d_6\equiv -1\pmod3\mbox{ and }k_6\equiv \pm 2\pmod9
\end{cases}$}
\end{equation*}
\item For $(1, \cdot,v\neq 0,0)$ we apply (\ref{eqn:cong}) in (\ref{Sec4_4}) obtaining:
\begin{equation*}
\resizebox{.9\hsize}{!}{$
w_D\equiv\begin{cases}
\ |d_6|k_6\pmod3&\mbox{when }k_6\equiv \pm 2\cdot 4^{v-1}\pmod9\\
\ |d_6|k_6\pmod3&\mbox{when }d_6\equiv (-1)^{m+1}\pmod3\mbox{ and }k_6\equiv \pm 1\cdot 4^{v-1}\pmod9\\
-|d_6|k_6\pmod3&\mbox{when }d_6\equiv (-1)^m\pmod3\mbox{ and }k_6\equiv \pm 1\cdot 4^{v-1}\pmod9\\
-|d_6|k_6\pmod3&\mbox{when }k_6\equiv \pm 4\cdot 4^{v-1}\pmod9
\end{cases}$}
\end{equation*}
\end{enumerate}
\item For $(0,\cdot,v\neq 0,0)$ we apply (\ref{eqn:cong}) in (\ref{Sec4_5}) obtaining:
\begin{equation*}
\resizebox{.9\hsize}{!}{$
w_D\equiv\begin{cases}
|d_6|k_6\pmod3&\mbox{when }k_6\equiv \pm 2^{2v-2}\pmod9\mbox{ and }d_6\equiv \pm 1\pmod{12}\\
-|d_6|k_6\pmod3&\mbox{when }k_6\equiv \pm 2^{2v-2}\pmod9\mbox{ and }d_6\equiv \pm 5\pmod{12}\\
(-1)^{m}|d_6|k_6\pmod3&\mbox{when }k_6\equiv \pm 2^{2v-1}\pmod9\mbox{ and }d_6\equiv 1,5\pmod{12}\\
(-1)^{m+1}|d_6|k_6\pmod3&\mbox{when }k_6\equiv \pm 2^{2v-1}\pmod9\mbox{ and }d_6\equiv -1,-5\pmod{12}\\
(-1)^{m+1}|d_6|k_6\pmod3&\mbox{when }k_6\equiv \pm 2^{2v}\pmod9\mbox{ and }d_6\equiv 1, 5\pmod{12}\\
(-1)^{m}|d_6|k_6\pmod3&\mbox{when }k_6\equiv \pm 2^{2v}\pmod9\mbox{ and }d_6\equiv -1,-5\pmod{12}
\end{cases}$}
\end{equation*}
\end{enumerate}
Now, we determine for each criterion when $w_D=-1$ obtaining the table above.
\end{proof}

Finally, we give the proof for Corollary \ref{cor:paperi}.
\begin{proof}
When $k=1$ then signatures are $(\cdot,\cdot,0,0)$, then we apply Theorem \ref{thm:table}, so we get four cases:
\begin{enumerate}
\item When $(0,0,0,0)$, $|d|=|d_6|\equiv -1\pmod3$, which implies $|d|\equiv -1,2,-4\pmod9$.
\item When $(0,1,0,0)$, $d=3d_6$ and $|d_6|\equiv -1\pmod3$, then $|d|=3|d_6|\equiv -3,6,-12\pmod9$, so $|d|\equiv 6\pmod9$.
\item When $(1,0,0,0)$, $d=2d_6$ and $|d_6|\equiv 1\pmod3$, so $|d|\equiv 2,-4,-1\pmod9$.
\item When $(1,1,0,0)$, $d=6d_6$ and $|d_6|\equiv 1\pmod3$, then $|d|\equiv 6,24,-12\pmod9$, so $|d|\equiv 6\pmod9$.
\end{enumerate}
\end{proof}

\section*{Acknowledgement}
\thispagestyle{empty}
The first author was partially supported by SNI - CONACYT. The third author was supported by the PhD Grant 783130 CONACYT - Government of Mexico.

\end{document}